\documentclass[a4paper, 12pt]{amsart}

\usepackage{amscd,amsthm,amsfonts,amssymb,amsmath,esint}
\usepackage[colorlinks=true]{hyperref}
\usepackage{fullpage}
\usepackage[all]{xy}
\usepackage[margin=1.3in]{geometry}
\usepackage{mathrsfs}
\usepackage{tikz-cd}
\usepackage{enumitem}

     \newcommand{\BF}{{\mathbb {F}}}

    \newcommand{\BQ}{{\mathbb {Q}}}

     \newcommand{\BZ}{{\mathbb {Z}}}

    \newcommand{\CG}{{\mathcal {G}}}

     \newcommand{\fp}{{\mathfrak{p}}}

    \newcommand{\Gal}{{\mathrm{Gal}}}

     \newcommand{\rank}{{\mathrm{rank}}}

    \theoremstyle{plain}
    \newtheorem{theorem}{Theorem}[section]
    \newtheorem{thm}{theorem}[section] 
    \newtheorem{corollary}[thm]{Corollary}
    \newtheorem{lemma}[thm]{Lemma}  
     \newtheorem{definition}[thm]{Definition}
        
\newtheorem{example}[thm]{Example}

    \numberwithin{equation}{section}

\begin{document}

\title {stabilization on ideal class groups in potential cyclic towers}

\author{Jianing Li}

\maketitle
\begin{abstract}
Let $p$ be a prime and let $F$ be a number field. 
Consider a Galois extension $K/F$ with Galois group $H\rtimes \Delta$ where $H\cong \BZ_p$ or $\BZ/p^d\BZ$, and $\Delta$ is an arbitrary Galois group. The subfields fixed
 by $H^{p^n} \rtimes \Delta$ $(n=0,1,\cdots)$ form a tower which we call it a potential cyclic $p$-tower in this paper. A radical $p$-tower is a typical example, say $\BQ\subset \BQ(\sqrt[p]{a})\subset  \BQ(\sqrt[p^2]{a})\subset \cdots$ where $a\in \BZ$. 
 We extend the stabilization result of Fukuda in Iwasawa theory on $p$-class groups in cyclic $p$-towers to potential cyclic $p$-towers. We also extend Iwasawa's class number formula in $\BZ_p$-extensions to potential $\BZ_p$-extensions. 

\end{abstract}

\section{Introduction}
Let $p$ be a prime number and let $F$ be a number field. Let $F_\infty/F$ be a $\BZ_p$-extension with $n$-th layer $F_n$. The classical Iwasawa theory studies the $p$-class groups of $F_n$ when $n$ varies. The stationary result in Iwasawa theory proved by Fukuda \cite{Fuk94}  states that if the $p$-class groups of $F_{n_0}$ and $F_{n_0+m}$ are isomorphic for $m=1$, then the same holds for each $m\geq 1$; here $n_0$ is some integer such that $F_\infty/F_{n_0}$ is totally ramified at every ramified prime. Using a similar argument, 
a stabilization result can be proved in a finite cyclic $p$-tower; see \cite[Proposition 3.2]{LOXZ22}. It is used there to explicitly compute certain $p$-class groups of fields in radical towers, like $\BQ(\mu_{p^2}, \sqrt[p^n]{a})$ with $n$ varies, which are not Galois extensions in general.  
It motivates the author to think whether there are stationary results for $p$-class groups in non-cyclic radical towers, say $\BQ\subset \BQ(\sqrt[p]{a}) \subset \BQ(\sqrt[p^2]{a}) \subset \cdots $, where $a$ is an integer. In this paper they will be called potential cyclic $p$-towers, and we 
 prove a stationary result in such towers and obtain an "Iwasawa like" class number formula for such a tower.

\begin{definition}[Potential cyclic $p$-tower]\label{def:main}
Let $p$ be a prime number. Let $F$ be a number field.
Let $K/F$ be a Galois extension and let $G=\Gal(K/F)$. Assume that $G=H\rtimes \Delta$ is a semidirect product of two subgroups ($H$ being normal), and that $H\cong \BZ/p^d\BZ$ or $H\cong \BZ_p$, while $\Delta$ may be an arbitrary (finite or infinite) Galois group. Let $F_n$ be the intermediate field of $K/F$ fixed by the group $H^{p^n}\rtimes \Delta$, so $F_0=F$. We then get a tower $F_0\subset F_1\subset \cdots \subset F_d$ ($d=\infty$ in the case $H=\BZ_p$) with $[F_{n+1}:F_n]=p$ for each $n$, and we say this tower is a \textit{potential cyclic p-tower} (in the case $H\cong\BZ_p$, we may call it a \textit{potential  $\BZ_p$-tower}). 
\end{definition}

\begin{example}\label{exa:potential_cyclic}
\begin{enumerate}[label=\upshape(\arabic*)]
\item $\BZ_p$-extensions and $\BZ/p^d\BZ$-extensions clearly give potential cyclic $p$-towers.
\item Suppose $a\in F$ such that $x^p-a$ is an irreducible polynomial. Then 
the radical tower $F_n=F(\sqrt[p^n]{a})$ ($n=0,1,\cdots, \infty$) is a potential cyclic p-tower, with $K=F_\infty(\mu_{p^\infty})$. A particular interesting example is $\BQ\subset \BQ(\sqrt[p]{p}) \subset  \BQ(\sqrt[p^2]{p}) \subset \cdots \subset\BQ(\sqrt[p^\infty]{p})$, in which every subextension is not Galois.
\end{enumerate}
\end{example}

To obtain results on class groups, we make the following hypothesis on ramifications:

\textbf{RamHyp:}
Let $K/F$ and $\{F_n\}$ be as in a potential cyclic $p$-tower.
If $v$ is a prime of $F$ which is ramified in $K/F$, we assume that the inertia group $I_v(K/F)$ of $v$ is conjugate to a subgroup of $G$ of the form $H\rtimes \Delta_v$, where $\Delta_v$ is a subgroup of $\Delta$. Moreover, assume that there exists a prime $v$ which is totally ramified in $K/F$, in other words $I_v(K/F)=G$.
 
 \begin{example}
 \begin{enumerate}[label=\upshape(\arabic*)]
 \item In a $\BZ_p$-extension $F_\infty/F$, \textbf{RamHyp} holds for $F_\infty/F_{n_0}$ when $n_0$ is sufficiently large.
 \item Regarding Example~\ref{exa:potential_cyclic}(2), we consider the $2$-dimensional $p$-adic tower $F_{n,m}=F(\sqrt[p^n]{a},\mu_{p^m})$ and let $K=F_{\infty,\infty}$. Then one can show that the potential cyclic tower arising from $K/F_{N,N}$ satisfies \textbf{RamHyp} when $N$ is sufficiently large. 
 \end{enumerate}
 \end{example}

If $E$ is any number field, let $A_E$ denote the $p$-class group (i.e. $p$-Sylow subgroup of the class group) of $E$.  Our first main result is the following:

\begin{theorem}\label{thm:main}
Let $\{F_n\}$ be a potential cyclic $p$-tower satisfying  \textbf{RamHyp}. 
 If $A_{F_1}\cong A_{F_0}$, then $A_{F_n}\cong A_{F_0}$ for $n\geq 1$. More generally, for each $k\geq 1$, if $A_{F_1}/p^k A_{F_1} \cong A_{F_0}/p^k A_{F_0}$, then $A_{F_n}/p^k A_{F_n} \cong A_{F_0}/p^k A_{F_0}$  for $n\geq 1$.
\end{theorem}
By \textbf{RamHyp}, the norm map $A_{F_n}\to A_{F_{n-1}}$ is always surjective for each $n$, so the condition $A_{F_1}\cong A_{F_{0}}$ is equivalent to $|A_{F_1}| = |A_{F_{0}}|$. For a finite abelian group $A$, recall that its $p^i$-rank $(i\geq 1)$ is defined as $\rank_{p^i}A := \mathrm{dim}_{\BF_p} (p^{i-1}A)/(p^iA)$.  A direct consequence is the following:
\begin{corollary}
If $|A_{F_1}|=|A_{F_0}|$, then $|A_{F_n}|=|A_{F_0}|$. More generally, for each $k\geq 1$, if $\rank_{p^i}A_{F_1}=\rank_{p^i}A_{F_0}$ holds for every $i\leq k$, then the same equalities hold for $A_{F_n}$ and $A_{F_0}$ for each $n\geq 1$.
\end{corollary}

Fix an integer $m\geq 0$, then the tower $F_n= \BQ(\mu_{p^m},\sqrt[p^n]{p})$ ($n=0,1,\cdots$) satisfies (see \cite[Lemma 4.1]{LOXZ22}) the condition in Theorem~\ref{thm:main} with $K=\BQ(\mu_{p^\infty}, \sqrt[p^\infty]{p})$. If $p$ is a regular prime, it is easy to show that $A_{F_n}$ is trivial for each $n$ and each $m$. 
Let $p=37$, the smallest irregular prime with $A_{\BQ(\mu_p)}=\BZ/p\BZ$
; then it is computed by McCallum and Sharifi that $A_{\BQ(\sqrt[p]{p})}=0$ and $A_{\BQ(\mu_p,\sqrt[p]{p})}=\BZ/p\BZ$; see \cite{MS03}. It follows from Theorem~\ref{thm:main} we indeed have $A_{\BQ(\mu_p,\sqrt[p^n]{p})}=\BZ/p\BZ$ and $A_{\BQ(\sqrt[p^n]{p})}=0$ for each $n\geq 1$. Other explicit examples can also be found in \cite{LOXZ22} and whose proofs can be simplified using Theorem~\ref{thm:main}.

Our second main result is a generalization of Iwasawa's classical class number formula (see \cite{Iwa58}) in a $\BZ_p$-extension. 
\begin{theorem}\label{thm:formula}
Let $\{F_n\}$ be a potential cyclic $p$-tower satisfying  \textbf{RamHyp}. 
 Let $e_n\geq 0$ such that $p^{e_n}=|A_{F_n}|$.
Then there exit non-negative integers $\mu, \lambda, \nu$ such that 
$e_n= \mu p^n +\lambda n +\nu$ when $n$ is sufficiently large.
\end{theorem}
This has been proved by Lei \cite{Lei17} under the following three conditions (1): $\Delta \cong \BZ_p$; (2): there is only one prime $\fp$ of $F$ lying above $p$; (3): $\fp$ is totally ramified in $K/F$. 

We mention a related result of Caputo and Nuccio \cite{CN23}.
Let $k$ be an imaginary quadratic field, let $K/k$ be the anti-cyclotomic $\BZ_p$-extension and suppose $p\neq 2$. Then $\Gal(K/\BQ) \cong \BZ_p\rtimes \BZ/2\BZ:=H\rtimes \Delta$. The subgroups $H^{p^n}\rtimes \Delta$ correspond to a potential $\BZ_p$-tower $\BQ=F_0\subset F_1 \subset \cdots F_\infty$, which are called "fake $\BZ_p$-extensions of dihedral type" in \cite{CN23}. They prove an "Iwasawa like" class number formula but the $\mu, \nu$-invariants are only in $\BZ[1/2]$ (see \cite[Theorem 4.6]{CN23}). We remark that Theorem~\ref{thm:formula} can not directly apply to this situation since the \textbf{RamHyp} does not hold in general.

\subsection*{Acknowledgment.}
The author thanks Yuan Ren and Lim Meng-Fai for helpful discussions on this paper. The author thanks the referee for helpful comments and suggestions.
The author is partially supported by NSFC (No. 12201347), by Shandong Provincial Natural Science Foundation (ZR202111170066), by the Fundamental Research Funds for the Central Universities, and by the Taishan Young Scholar Program (Grant No. tsqn202211043).

\section{The proof}
The idea of the proofs is from Iwasawa's theory of $\BZ_p$-extension (see \cite{Iwa58} or \cite[Chapter 13]{Was97}). 
\subsection{Descent in General Galois extension}

Let $p$ be a prime number and let $F$ be a number field. Let $K/F$ be an arbitrary Galois extension in this subsection. Write $G=\Gal(K/F)$. Let $L$ be the $p$-Hilbert class field of $K$, i.e., the maximal unramified pro-$p$-extension of $K$.  Since $K/F$ is a Galois extension, $L/F$ is also Galois.

\textbf{Notation}:
\begin{itemize}
\item 
Write $\CG=\Gal(L/F)$ and $X=\Gal(L/K)$. The group $G$ acts on $X$ by lifting conjugation. Say $g\in G, x\in X$, then $g\cdot x=\tilde{g}x\tilde{g}^{-1}$, where $\tilde{g}\in \CG$ is any lifting of $g$.

\item Let $\BZ_p[G]$ denote the (completed) group ring, so that $X$ is a $\BZ_p[G]$-module. Let $I_G$ denote the argumentation ideal of $\BZ_p[G]$, so $I_G$ is the $\BZ_p$-linear combination of $\{g-1: g\in G\}$. Then the closure of $I_GX$ (written as $I_G X$ by abuse of notation) is a compact $\BZ_p[G]$-submodule of $X$ and is a normal subgroup of $\CG$. 
\item 
If $V$ is a closed subgroup of $\CG$, let $[V,V]$ denote the closure of the commutator subgroup of $V$; and let $V'$ denote the closure of the image of $V$ in $\CG/I_GX$. In particular, $X'=X/I_GX$, on which $G$ acts trivially.

\item If $T\subset \CG$, the notation $\langle T\rangle$ denotes the closure of the subgroup of $\CG$ generated by $T$.

\item  If $E/F$ is a Galois extension of number fields and $w$ is a prime of $E$, let $I_w(E/F)$ denote the inertia group $w$.

\item Let $v_1, \cdots, v_r$ (here $r\leq \infty$, but in our further application $r<\infty$ always) be the primes of $F$ which are ramified in $K$. For each $i$, choose a prime $w_i$ of $L$ lying above $v_i$.

\end{itemize}

\begin{lemma}\label{lem:general}
Assume that $w=w_1$ is totally ramified in $K/F$. Write $I_{w} =I_w(L/F)$. Then we have:
\begin{enumerate}[label=\upshape(\arabic*)]
\item  $I_{w}\cap X = 0$, $\CG/X \cong G$, and $I_{w}\cong G$ under this isomorphism. Therefore, $\CG=X\rtimes I_{w}$.
\item $[\CG,\CG]=I_G X \rtimes [I_w,I_w]$.
\item $I'_w$ is normal in $\CG'$ and $\CG'\cong X'\times I'_w$.
\item $A_F\cong \Gal(L(F)/F) \cong \CG/\langle I_GX,I_w, I_{w_2},\cdots, I_{w_r}\rangle \cong \CG'/( I'_w \langle I'_{w_2},\cdots, I'_{w_r}\rangle)$, where $L(F)$ is the $p$-Hilbert class field of $F$.
\item Let $Y=X\cap  \langle I_GX, I_{w}, I_{w_2},\cdots, I_{w_r}\rangle$. Then the inclusion $X\to \CG$ induces an isomorphism $X/Y \cong A_F$. \end{enumerate}
\end{lemma}
Before proof, we remark that  these results are independent of the ways in which the primes $w_i$ of $L$ lying above $v_i$ are chosen.

\begin{proof}[Proof]\begin{enumerate}
\item This is obvious, since $v$ is totally ramified in $K/F$ and $L/K$ is unramified. 
\item 
Since $(g-1)x=\tilde{g}x\tilde{g}^{-1}x^{-1}$ for $g-1 \in I_G, x\in X$, the commutator group of $\CG$ contains the subgroup
$\langle I_G X, [I_{w},I_{w}]\rangle$ which is equal to $I_G X \rtimes [I_w,I_w]$ as $I_G X$ is normal in $\CG$. On the other hand, since $\CG=X\rtimes I_w$ and since $G\cong I_w$ acts on $X'$ trivially, we have isomorphisms
($X\rtimes I_w)/(I_G X \rtimes [I_w,I_w]) \cong (X'\times I_w)/[I_w,I_w]$, the latter being abelian clearly. This proves (2).
\item By (1) and the definition of $X'$, we have $X'\cap I'_w=0$. Hence we have $\CG'=X'\rtimes I'_w$. Since the action of $\CG$ and hence $I_w$ on $X'$ is trivial, it follows that $I'_w$ is normal in $\CG'$ and therefore $X'\rtimes I'_w =X'\times I'_w$.
\item The first isomorphism is by class field theory. The second is from the definiton of $L(F)$ and (2); the third is from (3).
\item This is obvious.
\end{enumerate}
\end{proof}

\subsection{Descent in Potential Cyclic $p$-towers}
Let $K/F$, $\{F_n\}$, $G=H\rtimes\Delta$ be as in a potential Cyclic $p$-tower; see Definition~\ref{def:main}. Assume it satisfies \textbf{RamHyp}. 
So $H\cong \BZ/p^d\BZ$ or $H\cong \BZ_p$, and we fix a (topological) generator $h$ of $H$. 
We shall apply Lemma~\ref{lem:general} to $K/F_n$ whose Galois group is $H^{p^n}\rtimes \Delta$;  we first treat the case $K/F=K/F_0$.

Assume $v_1,\cdots, v_r$ are the primes of $F$ ramified in $K$, and assume that $v=v_1$ is totally ramified in $K$. By \textbf{RamHyp}, for each $i$, we can choose some prime $w_i$ of $L$ such that (writing $w=w_1$)
\begin{equation}\label{eq:inertia_iso}
I_{w_i}=I_{w_i}(L/F)  \cong H\times \Delta_i \quad \text{ and } \quad I_{w}\cong G=H\times \Delta
\end{equation} under the isomorphism $\CG/X\cong G$ in Lemma~\ref{lem:general}(1).  
If $g\in G$, let $\tilde{g}_{i} \in I_{w_i}$ be the inverse image of $g$, when it exists, under \eqref{eq:inertia_iso}. Then for $1\leq i \leq r$, we have 
$\tilde{h}_i\equiv \tilde{h}_{1} \bmod X$ and for each $\delta \in \Delta_i$, $\tilde{\delta}_i \equiv \tilde{\delta}_{1} \bmod X$. 
Now put 
\begin{gather}\label{eq:ab}
a_i:=\tilde{h}^{-1}_1\tilde{h}_i\in X,\quad b_{\delta,i} := \tilde{\delta}^{-1}_{1}\tilde{\delta}_i\in X,
\end{gather}
Define subgroups of $X$ by 
\[A=\langle a_2,\cdots, a_r\rangle \text{ and }B=\langle \{b_{\delta,i}: i=1,\cdots, r, \delta\in \Delta_{i}\}\rangle.\]
It follows that
\[
\langle I_{w_1},\cdots, I_{w_r}\rangle = \langle I_{w}, A,B\rangle,
\]
and that as subgroups of $\CG'=\CG/I_G X$, 
\begin{equation}\label{eq:inert2}
\langle I'_{w_1},\cdots, I'_{w_r}\rangle = \langle I'_{w}, A',B'\rangle = I'_{w}\langle A',B'\rangle.
\end{equation}
The last equality holds because $I'_{w}$ is normal in $\CG'$ by Lemma~\ref{lem:general}(3). Since $I'_w \cap X'$ is trivial by Lemma~\ref{lem:general}(3), we obtain 
\begin{equation}
X'\cap \langle I'_{w_1},\cdots, I'_{w_r}\rangle = \langle A',B'\rangle.
\end{equation}
It follows that 
\begin{equation}
X\cap \langle I_G X, I_{w_1},\cdots, I_{w_r} \rangle =\langle I_G X, A,B\rangle.
\end{equation}
We remark that the left hand side is \textit{a prior} a $G$-submodule of $X$, so is the right hand side.
By Lemma~\ref{lem:general}(5) we have
\begin{equation}
A_{F_0} \cong X/\langle I_G X, A,B\rangle.
\end{equation}

Now we apply the above argument and Lemma~\ref{lem:general} to $K/F_n$, so $L$ remains unchanged, $\CG$ becomes $\CG_n=\Gal(L/F_n)$, and $G$ becomes $G_n:=\Gal(L/F_n)$. The inertia groups $I_{w_i}(L/F_n)=I_{w_i}(L/F)\cap \Gal(L/F_n)$ and now maps isomorphically to $H^{p^n}\rtimes \Delta_i$. So the elements $b_{\delta,i}$ in \eqref{eq:ab} are unchanged, but  those $a_i$ change to $\tilde{h}_1^{-p^n}\tilde{h}_i^{p^n}$.
Let $$\omega_n =\frac{h^{p^n}-1}{h-1} \in \BZ_p[H].$$ Then one directly computes that 
\begin{equation}
\tilde{h}_1^{-p^n}\tilde{h}_i^{p^n} = \omega_n a_i,
\end{equation}
Hence as subgroups of $X':=X/I_{G_n} X$, we have 
$$X'\cap \langle I'_{w_1}(L/F_n),\cdots, I'_{w_r}(L/F_n)\rangle = X'\cap (I'_{w_1}(L/F_n) \langle \omega_n A', B'\rangle)=\langle \omega_n A',B'\rangle;$$
here the first equality uses that $I'_{w_1}(L/F_n)$ is normal in $\Gal(L/F_n)$ by Lemma~\ref{lem:general}, and the second is obvious.
Then as subgroups of $X$, we have
$$X\cap \langle I_{G_n} X, I_{w_1}(L/F_n),\cdots, I_{w_r}(L/F_n) \rangle =\langle I_{G_n} X, \omega_nA,B\rangle.$$
By Lemma~\ref{lem:general}(5) we have that for each $n\geq 0$,
\[
A_{F_n} \cong X/\langle I_{G_n} X, \omega_n A, B\rangle.
\]
\begin{lemma}
As a left $\BZ_p[H^{p^n}]$-module, $I_{G_n}$ is generated by $h^{p^n}-1=\omega_n (h-1)$ and $I_\Delta$, where $I_\Delta$ is the $\BZ_p$-linear combination of $\{\delta-1:\delta\in \Delta\}$. 
\end{lemma}
\begin{proof}
It suffices to prove the case $n=0$. 
If $h^i\delta \in G$, then $h^i\delta-1=h^i\delta-h^i+h^i-1=h^i(\delta-1)+(h^i-1)=h^i(\delta-1)+\frac{h^i-1}{h-1}(h-1)$. Since $I_{G_0}$ is the linear combination of $h^i\delta-1$, this completes the proof.
\end{proof}

Now consider the following two $\BZ_p[H]$-submodules of $X$, both independent of $n$: 
\begin{equation}
C:=\langle (h-1)X, A\rangle_{H-\mathrm{mod}}, \quad D:=\langle I_\Delta X, B\rangle_{H-\mathrm{mod}}.
\end{equation}
Here the notation $\langle T  \rangle_{H-\mathrm{mod}}$ means the $\BZ_p[H]$-submodule of $X$ generated by $T$ for $T\subset X$.
Then we obtain the following isomorphisms  for each $n\geq 0$
\begin{equation}\label{eq:clgp_final_expression}
A_{F_n} \cong X/(\omega_n C+D) \cong\overline{X}/\omega_n \overline{C}.
\end{equation}
Here we write $\overline{X} := X/D$  and $\overline{C} := (C+D)/D$; both are $\BZ_p[H]$-modules and are independent of $n$.

\begin{theorem}\label{thm:finite_torsion}
In the case $H\cong \BZ_p$, 
the $\BZ_p[H]$-module $\overline{X}$ is finitely generated and torsion.
\end{theorem}
\begin{proof}

We turn to the case $H\cong \BZ_p$. 
We first show that $\overline{X}$ is finitely generated over $\BZ_p[H]$. It suffices to show that $\overline{X}/(p, h-1)\overline{X}$ is finite by the compact version of Nakayama's lemma; see \cite[Lemma 13.16]{Was97}. Since $\omega_1 \in (p,h-1)$ the module $\overline{X}/(p, h-1)\overline{X}$ is a quotient of $\overline{X}/\omega_1 \overline{X}$ which is a quotient of $\overline{X}/\omega_1 \overline{C}\cong A_{F_1}$, which is finite. 

According to  the  structure theorem of $\BZ_p[H]$-modules (see \cite[\S 13.2]{Was97}), there is a direct summand $(\BZ_p[H])^r$ in $\overline{X}$ (up to a finite group). But $\BZ_p[H]/\omega_n \BZ_p[H]$ is infinite, so $r=0$ and therefore $\overline{X}$ is torsion over $\BZ_p[H]$.
\end{proof}

\begin{proof}[Proof of Theorem~\ref{thm:formula}]
Classical arguments in Iwasawa theory (see for example \cite[Lemma 13.18, Proposition 13.19, Lemma 13.21]{Was97}) shows that if $Y'\subset Y$ are finitely generated torsion $\BZ_p[H]$-modules such that $Y/\omega_n Y'$ is finite for each $n$, then $|Y/\omega_n Y'|= \mu p^n+\lambda n +\nu$ when $n$ is sufficiently large for integers $\mu, \lambda, \nu$. Therefore Theorem~\ref{thm:formula} follows from \eqref{eq:clgp_final_expression} and Theorem~\ref{thm:finite_torsion}.
\end{proof}

\begin{proof}[Proof of Theorem~\ref{thm:main}]
We only need to prove the second statement. By \eqref{eq:clgp_final_expression}, $A_{F_n}/p^kA_{F_n} \cong \overline{X}/\left(p^k \overline{X}+\omega_n \overline{C}\right)$ for each $n$. 
Then the condition of the theorem implies that $p^k \overline{X}+\omega_1 \overline{C} = p^k \overline{X}+ \overline{C}$. Since $\omega$ is in the maximal ideal of $\BZ_p[H]$, by Nakayama's lemma we have $\overline{C}\subset p^k \overline{X}$. Therefore, $A_{F_n}/p^kA_{F_n} \cong \overline{X}/p^k \overline{X}$ for each $n$. This completes the proof of Theorem~\ref{thm:main}.\end{proof}

\bigskip

\noindent 
Research Center for Mathematics and Interdisciplinary Sciences, Shandong University, Qingdao 266237, PR China 

\

\noindent
Frontiers Science Center for Nonlinear Expectations, Ministry of Education, Qingdao 266237, PR China

{\it lijn@sdu.edu.cn}

\end{document}